\documentclass[11pt, epsfig]{article}
\usepackage{epsfig, amsmath, amssymb, amsthm, times}

\usepackage[T1]{fontenc}
\usepackage{hyperref}

\usepackage{algorithmicx}
\algrenewcommand{\alglinenumber}[1]{\footnotesize Step #1:}

\usepackage[dvipsnames]{xcolor} 

\newcommand{\bel}{\begin{equation}\label}
\newcommand{\ee}{\end{equation}}

\def\R{{\mathbb R}}
\def\C{{\mathbb C}}

\def\<{\langle}
\def\>{\rangle}

\def\V{\mathcal V}
\def\W{\mathcal W}
\def\A{\mathcal A}
\def\B{\mathcal B}
\def\C{\mathcal C}

\def\x{\mathbf x}
\def\0{\mathbf 0}

\newtheorem{theorem}{Theorem}[section]
\newtheorem{proposition}[theorem]{Proposition}

\newtheorem{lemma}[theorem]{Lemma}

\theoremstyle{definition}

\newtheorem{problem}{Problem}

\title{Optimality of the recursive Neyman allocation}
\author{Jacek Weso\l owski\thanks{Professor at the Warsaw University of Technology, Warsaw, Poland and Senior Research Statistician at Statistics Poland, Warsaw, Poland,   \href{mailto:wesolo@mini.pw.edu.pl}{wesolo@mini.pw.edu.pl}},  Robert Wieczorkowski\thanks{Senior Research Statistician at Statistics Poland, Warsaw, Poland,  \href{mailto:R.Wieczorkowski@stat.gov.pl}{R.Wieczorkowski@stat.gov.pl}}, Wojciech W\'ojciak\thanks{PhD Student at the Warsaw University of Technology, Warsaw, Poland,  \href{mailto:wojciech.wojciak@gmail.com}{wojciech.wojciak@gmail.com}}}
\date{}

\begin{document}
	\maketitle
	\begin{abstract}
		We derive a formula for the optimal sample allocation in a general stratified scheme under upper bounds on the sample strata-sizes. Such a general scheme includes SRSWOR within strata as a special case. The solution is given in terms of $\V$-allocation with $\V$ being the set of {\em take-all} strata.  We use $\V$-allocation to give a formal proof of optimality of the popular recursive Neyman algorithm, {\em rNa}. This approach is convenient also for a quick proof of optimality of the algorithm of Stenger and Gabler (2005), {\em SGa}, as well as of its modification, {\em coma}, we propose here. Finally, we compare running times of {\em rNa}, {\em SGa} and {\em coma}. Ready-to-use R-implementations of these  algorithms are available on CRAN repository at https://cran.r-project.org/web/packages/stratallo.
	\end{abstract}
	
	\section{Introduction}
	Optimal sample allocation in stratified sampling scheme is one of the basic problems of survey methodology. An abundant body of literature, going back to the classical optimal solution of Tchuprov (1923) and Neyman (1934)  for the case of simple random sampling without replacement (SRSWOR) design in each stratum, is devoted to this issue. In recent years, there has been a growing interest in more refined allocation methods, mostly based on non-linear programming (NLP), see, e.g.  Valliant, Dever and Kreuter (2018) and references therein. Except the NLP methods, which give only approximate solutions (typically sufficiently precise), a number of recursive  allocation methods have been developed over the years.  The recursive Neyman algorithm, described in Remark 12.7.1 in S\"arndal, Swensson, Wretman (1992), seems to be popular among practitioners. For more recent recursive methods see e.g. Kadane (2005), Stenger and Gabler (2005), Gabler, Ganninger and M\"unnich (2012), Friedrich, M\"unnich, de Vries and Wagner (2015) or Wright (2017, 2020). Another non-recursive method, based on fixed point iterations, was proposed in M\"unnich, Wagner and Sachs (2012). In this paper we are concerned with recursive methods. 
	
	Let $U$ be a population of size $N$. For a study variable $\mathcal{Y}$ defined on $U$ we write $y_k=\mathcal{Y}(k)$ to denote its value for unit $k\in U$. The parameter of interest is the total of variable $\mathcal{Y}$ in $U$, $t_{\mathcal{Y}}=\sum_{k\in U}\,y_k$. 
	
	To estimate $t_{\mathcal Y}$ we consider the stratified SRSWOR design, under which the population $U$ is stratified, i.e. $U=\bigcup_{w\in\mathcal W}\,U_w$, where the strata $U_w$, $w\in\W$, are disjoint and non-empty and $\W$ denotes the set of strata labels. We denote by $N_w$ the size of $U_w$, $w\in\W$. A sample $\mathcal S_w$ of size $n_w$ is drawn according to SRSWOR from $U_w$, $w\in\W$. The draws  between strata are independent.  The $\pi$-estimator of  $t_{\mathcal Y}$ is given by $\hat{t}_{st}=\sum_{w\in\W}\tfrac{N_w}{n_w}\,\sum_{k\in\mathcal{S}_w}\,y_k$. It is design unbiased with variance 
	\begin{equation}\label{d2st}
	D^2_{st}=\sum_{w\in \W}\,\tfrac{d_w^2}{n_w}-\sum_{w\in \W}\,\tfrac{d_w^2}{N_w},
	\end{equation}
	where $d_w=N_wS_w$ and  $S_w^2=\tfrac{1}{N_w-1}\sum_{k\in U_w}\,(y_k-\bar{y}_w)^2$ with $\bar{y}_w=\tfrac{1}{N_w}\sum_{k\in U_w}\,y_k$, $w\in\W$.
	
	The problem of optimal sample allocation lies in the determination of the allocation vector $(n_w)_{w\in \W}$ that minimizes \eqref{d2st} subject to 
	\begin{equation}\label{cons1}
	\sum_{w\in \W}\,n_w=n,
	\end{equation}
	with  $n\le N=\sum_{w\in \W}\,N_w$. 
	
	The classical solution to this problem, called traditionally the Tchuprov-Neyman allocation  (Tchuprov 1923; Neyman 1934), has the form
	\begin{equation}\label{Neso}
	n_w^*=n\,\tfrac{d_w}{d},\quad w\in \W,
	\end{equation} 
	where $d=\sum_{v\in \W}\,d_v$. 
	
	It is well-known that  $(n^*_w)_{w\in\W}$ given by \eqref{Neso} may not be feasible since it may violate the natural constraints
	\begin{equation}\label{upper}n_w\le N_w \qquad\mbox{for all}\qquad  w\in\W.\end{equation} 
	
	In other words,  \eqref{Neso} may over-allocate the sample is some strata. Thus, it has to be modified. In {\em Survey Methods and Practice} (2010), the authors write "... a census should be conducted in the over-allocated strata. The
	overall sample size resulting from such over-allocation will then be smaller than the original sample size,
	so the overall precision requirements might not be met. The solution is to increase the sample in the remaining strata where $n_w^*$ is smaller than $N_w$ using the surplus in the sample sizes obtained from the overallocated strata."  This idea is realized through the recursive Neyman algorithm (referred to by {\em rNa} in the sequel). It is a popular tool in everyday survey practice, though  its optimality remains an open question. 
	
	An alternative approach was introduced in Stenger and Gabler (2005), where the authors proposed another allocation algorithm (referred to by {\em SGa} in the sequel) and established its optimality. This algorithm can be descibed as follows: first, order strata labels in $\W$ with respect to non-increasing values of $S_w$, $w\in\W$; second, perform a  sequential search for the last {\em take-all} stratum in $\W$ ordered in  the previous step; third, compute the Tchuprov-Neyman allocation in the remaining strata. Gabler et al. (2012) extended this approach  to cover both, the upper and the lower bounds on the sample strata sizes and proposed an R-function, called {\em noptcond}, as implementation of their allocation procedure. In M\"unnich et al. (2012), the authors used the Karush-Kuhn-Tucker (referred to by KKT in the sequel) conditions to derive optimal allocation formula expressed in terms of, so called,  optimal Lagrange multiplier, being the root of a highly irregular function. They analyzed several fixed point iteration routines to speed up its calculation. Recently, integer-valued optimal allocation procedures have been developed in Friedrich at al. (2015) and Wright (2017).  
	
	In this paper, (a) we prove optimality of the  recursive Neyman algorithm, (b) we introduce a modification of the Stenger-Gabler algorithm  and prove its optimality, and (c) we compare  computational efficiency of the recursive Neyman algorithm, the Stenger-Gabler algorithm and its proposed modification.
	
	Actually, we consider a more general optimization scheme described in Problem 1.

	\begin{problem}\label{probl}
		{\em Given numbers $n, a_w, b_w>0$, $w \in \W$, minimize the objective function
			\begin{equation}\label{function}
			f(\x)=\sum_{w\in \W}\,\tfrac{a_w^2}{x_w},\quad \x=(x_w)_{w\in \W}
			\end{equation}
			subject to
			$$
			\sum_{w\in \W} x_w = n\qquad \mbox{and}\qquad 0<x_w\le b_w,\quad w\in \W.
			$$}
\end{problem}
	
	Problem 1 covers the case of stratified SRSWOR by assigning $a_w=d_w$, $b_w=N_w$, $w\in\W$. Clearly, it is feasible only if $n\le \sum_{w\in\W}\,b_w$. When $n=\sum_{w\in\W}\,b_w$, the solution is trivial and is equal to $(b_w)_{w\in\W}$. Therefore, we assume throughout this paper that $n<\sum_{w\in\W}\,b_w$.  Since Problem 1 is a convex optimization problem, its solution exists and is unique. It is identified  in Theorem \ref{gener} below. 
	
	For $\V \subseteq \W$ by $\x^\V=(x_w^\V)_{w\in \W}$ we denote the vector with entries 
	\begin{equation}\label{xop}x_w^\V=\left\{\begin{array}{ll}
		b_w, &\mathrm{for}\;  w\in \V, \\
		a_w\,s(\V), & \mathrm{for}\; w\not\in  \V,\end{array}\right.
	\end{equation}	
	where $s$ is a strictly positive function defined on proper subsets of $\W$ by   
	\begin{equation}\label{esv}
	s(\V)= \tfrac{n-\sum_{w\in\V}\,b_w}{\sum_{w\not\in \V}\,a_w},\qquad \V\subsetneq \W.
	\end{equation}
	
	We refer to $\x^\V$ by  $\V$-allocation. It turns out that the optimal allocation of the sample among strata is of the form \eqref{xop} for a unique subset $\V\subseteq \W$.
	
	\begin{theorem}
		\label{gener} The $\V$-allocation vector $\x^\V$ solves Problem \ref{probl} if and only if
		\begin{equation}\label{asop}
		\V=\left\{w\in \W:\,c_w\,s(\V)\ge 1\right\},
		\end{equation}
		where $c_w=\tfrac{a_w}{b_w}$, $w\in \W$.
	\end{theorem}

	The proof of Theorem \ref{gener}, based on the KKT conditions, is given in the Appendix. We use \eqref{asop}  to prove optimality of the recursive Neyman algorithm, {\em rNa}, in Section 2. In Section 3, following \eqref{asop}, we give a short proof of optimality of the Stenger-Gabler algorithm, {\em SGa} and introduce a modification, {\em coma}, for which we prove its optimality. In Section 4, we compare {\em rNa}, {\em SGa} and {\em coma} in terms of computational efficiency  (also with algorithms designed for optimal allocation with double-sided constraints: {\em noptcond} of Gabler et al. (2012) and {\em capacity scaling} of Friedrich et al. (2015)). As pointed out in M\"unnich et al. (2015), computational  efficiency of allocation algorithms becomes an issue "in cases with many strata or when the optimal allocation has to be applied repeatedly, such as in iterative solutions of  stratification problems". For the latter issue the reader is referred  to Dalenius and Hodges (1959), Lednicki and Wieczorkowski (2003), Guning and Horgan (2004), Kozak and Verma (2006) and Baillargeon and Rivest (2011). Moreover, in Section 4 we also compare variances of estimators based on the optimal intger-valued allocation (obtained e.g. with the {\em capacity scaling} algorithm) and of those based on the integer-rounded optimal allocation (obtained e.g. with the {\em rNa}).

	\section{The recursive Neyman algorithm}
	In this section, we prove that the {\em rNa} (Remark 12.7.1 in S\"arndal et al., 1992), leads to the allocation that minimizes  \eqref{d2st} under the constraints \eqref{cons1} and \eqref{upper}. For  a generalized setting of this minimization problem, Problem \ref{probl}, the algorithm {\em rNa} proceeds as follows:

\begin{algorithmic}[1]
	\State Let $\V_1 = \emptyset$, $r = 1$.
	\State Compute $s(\V_r)$ according to \eqref{esv}. 
	\State Let $R_r = \left\{w\in \W \setminus \V_r:\, c_w\, s(\V_r) \ge  1\right\}$. 
	\State If {$R_r =\emptyset$},  set $r^*=r$ and go to Step 5; 
	\Statex \hspace*{9.5pt} otherwise, set $\V_{r+1} = \V_r \cup R_r$, $r \gets r + 1$, and go to \footnotesize Step 2\normalsize.
	\State Return $\x^{\V_{r^*}}$ according to \eqref{xop}.
\end{algorithmic}

%
%
%
%

The numerical behavior of the {\em rNa} is illustrated in Table 1. 
\begin{table}\footnotesize
	\centering
	\begin{tabular}{||r|r|r|r|r|r||r|r|r|r|r|r||}
		\hline\hline
		$w$ & $c_w$ & $c_ws(\V_1)$  & $c_ws(\V_2)$  & $c_ws(\V_3)$ & $x_w^{\V_4}$ & $w$ & $c_w$ & $c_ws(\V_1)$  & $c_ws(\V_2)$  & $c_ws(\V_3)$ & $x_w^{\V_4}$ \\
		\hline
		1 &  0.33 & 0.062 & 0.1174 & 0.1303 & 130.3 &  11 & 2.37 & 0.444 & 0.8349 & 0.927 & 927.3 \\
		{\bf 2} & 2.65 & 0.480 & 0.9016 & {\bf 1.0011} & {\bf 1000} &  12 & 0.36 & 0.068 & 0.1282 & 0.1423 & 142.3 \\
		3 & 0.15 & 0.029 & 0.0543 & 0.0603 & 60.4 & 13 & 0.14 & 0.026 & 0.0493 & 0.0547 & 54.7 \\
		4 & 0.66 & 0.123 & 0.2316 & 0.2571 & 257.2 &  14 & 0.37 & 0.07 & 0.1316 & 0.1462 & 146.2  \\
		5 & 0.15 & 0.028 & 0.0519 & 0.0577 & 57.7 & {\bf 15} & 4.25 & 0.796 & {\bf 1.4967} & {\bf *} & {\bf 1000} \\
		{\bf 6} & 15.45 & {\bf 2.895} & {\bf *} & {\bf *} & {\bf 1000} & 16 & 0.39 & 0.074 & 0.1386 & 0.1539 & 153.9 \\
		7 & 1.49 & 0.279 & 0.5239 & 0.5817 & 581.9 & {\bf 17} & 10.21 & {\bf 1.913} & {\bf *} & {\bf *} & {\bf 1000} \\
		8 & 1.74 & 0.326 & 0.612 & 0.6796 & 679.7 & 18 & 0.10 & 0.018 & 0.0339 & 0.0376 & 37.6 \\
		9 & 0.30 & 0.056 & 0.1057 & 0.1173 & 117.3 & 19 & 0.23 & 0.044 & 0.0827 & 0.0918 & 91.8  \\
		10 & 0.93 & 0.174 & 0.3278 & 0.364 & 364.1 & 	20 & 0.51 & 0.095 & 0.1779 & 0.1975 & 197.6 \\
		\hline
	\end{tabular}
	\caption{\label{tab0:} 	\footnotesize For an artificial population with 20 strata, $n=8000$, $(c_w)_{w\in \W}$ generated as absolute values of independent Cauchy random variables and $b_w=1000$, $w\in\W$, the {\em rNa} assigns as {\em take-all} strata: $R_1 = \{6, 17\}$, $R_2 = \{15\}$, $R_3 = \{2\}$, and it stops at iteration 4 ($R_4 = \emptyset$) with $\V_4 = \{2, 6, 15, 17\}$. The optimal (non-integer) allocation (rounded to $0.1$)  is given in  columns labelled $x_w^{\V_4}$.} 
\end{table}
Before we proceed to prove that the {\em rNa} does indeed find the optimal solution to Problem \ref{probl}, we first derive in Lemma \ref{prop:ximo} a monotonicity property of function $s$, defined in \eqref{esv}. This property turns out to be a convenient tool in proving optimality of the {\em rNa} as well as the algorithms we consider in Section 3.

\begin{lemma}\label{prop:ximo} 
	Let $\A,\B\subset  \W$ be such that $\A \cap \B = \emptyset$ and $\A \cup \B\subsetneq \W$. Then
	\begin{equation}\label{ximo}
		s(\A \cup \B)\ge s(\A)\quad\mbox{if and only if}\quad s(\A)\sum_{w\in \B}\,a_w \geq \sum_{w\in \B}\,b_w.
	\end{equation}
\end{lemma}
\begin{proof}
	Clearly, for $\alpha, \beta, \delta \ge 0$ and $\gamma>0$, we have 
	\begin{equation}\label{genine}
	\tfrac{\alpha}{\gamma}\ge \tfrac{\alpha+\beta}{\gamma+\delta}\quad\mbox{if and only if}\quad \tfrac{\alpha+\beta}{\gamma+\delta} \delta \ge \beta.
	\end{equation}

	Denote $\C=\A\cup\B$. Take $\alpha=n-\sum_{w\in\C}\,b_w$, $\gamma=\sum_{w\not\in \C}\,a_w$, $\beta=\sum_{w \in \B}\,b_w$, $\delta=\sum_{w \in \B}\,a_w$. Thus, $\tfrac{\alpha}{\gamma}=s(\C)$ and  $\tfrac{\alpha+\beta}{\gamma+\delta}=s(\A)$. Then \eqref{ximo} is an immediate consequence of \eqref{genine}.
\end{proof}

\begin{theorem}\label{th_rna} 
	The algorithm rNa solves Problem \ref{probl}.
\end{theorem}
\begin{proof}
 According to \eqref{asop}, in order to prove that  $\x^{\V_{r^*}}$ in \eqref{xop} is the optimal allocation, we need  to show that  
 \begin{equation}\label{kkkk}
 w\in \V_{r^*}\quad\mbox{if and only if}\quad c_w\,s(\V_{r^*})\ge 1.
 \end{equation} 

For $r^*=1$ we have $\V_{r^*}=\emptyset$ and $R_1=\emptyset$, i.e. \eqref{kkkk} trivially holds. 

Consider  $r^*>1$. Since $n<\sum_{w\in \W}\,b_w$ we have $r^*\le K$, where $K$ is the number of strata.

{\em Sufficiency:} First, assume that $c_w\,s(\V_{r^*})\ge 1$ and $w\not\in \V_{r^*}$. Then, \footnotesize Step 4 \normalsize of {\em rNa} yields $c_w\,s(\V_{r^*})<1$ for $w\not \in \V_{r^*}$, which is  a contradiction. 

{\em Necessity:} By \footnotesize Step 3 \normalsize of {\em rNa}, we have $s(\V_r)\,a_w\geq b_w$, $w \in R_r$, for every $r \in \{1, \ldots, r^*-1\}$. Summing these inequalities over $w\in R_r$ we get the second inequality in \eqref{ximo} with $\A = V_r$, and $\B = R_r$. Since $\V_r \cup R_r\subsetneq \W$ and $\V_r\cap R_r=\emptyset$, by Lemma \ref{prop:ximo}, the first inequality in \eqref{ximo} follows. Consequently,
		\begin{equation}\label{svk} s(\V_1) \le \ldots \le s(\V_{r^*}).
		\end{equation}
		
		Now, assume that $w\in \V_{r^*}$. Thus, $w\in R_r$ for some $r\in\{1,\ldots,r^*-1\}$. Then, again using \footnotesize Step 3 \normalsize of {\em rNa}, we get $c_w s(\V_r)\geq 1$.  Consequently, \eqref{svk} yields $c_w s(\V_{r^*})  \geq 1$.
\end{proof}

\section{Revisiting the Stenger and Gabler methodology} 

Stenger and Gabler (2005, Lemma 1)  proposed another allocation algorithm  and proved its optimality. In contrast to the {\em rNa}, their algorithm is based on ordering the set of strata labels $\W$  with respect to $(c_w)_{w \in \W}$. In this section, we adjust Stenger and Gabler's algorithm to the general scheme of Problem \ref{probl} (such adjusted algorithm is  also referred to by {\em SGa}) and give a short proof of its optimality based on \eqref{asop}. We  also propose a modification  called {\em coma}. To describe both algoritms it is convenient to introduce the  notation: $\V_1 = \emptyset$ and $\V_r = \{1, \ldots, r-1\}$ for $r>1$. 

The algorithm {\em SGa} proceeds as follows: 

\begin{algorithmic}[1]
	\State Reorder $\W = \{1, \ldots, K\}$ according to $ c_1 \ge \ldots \ge c_K$.
	\State Let $r = 1$.
	\State Compute $s(\V_r)$ according to \eqref{esv}. 
	\State If $c_r s(\V_r) < 1$, set $r=r^*$ and go to \footnotesize Step 5\normalsize;
	\Statex \hspace*{9.5pt} otherwise, let $r \gets r + 1$, and go to \footnotesize Step 3\normalsize.	
	\State Return $\x^{\V_{r^*}}$ according to \eqref{xop}.
\end{algorithmic}

\begin{proposition}\label{prop:SGa} The algorithm SGa solves Problem \ref{probl}.
\end{proposition} 

\begin{proof}	According to \eqref{asop}, in order to prove that  $\x^{\V_{r^*}}$ in \eqref{xop} is the optimal allocation, we need  to show that
	\begin{equation}\label{kkkkk}
	r \in \V_{r^*}\quad\mbox{if and only if}\quad c_r\,s(\V_{r^*})\ge 1.
	\end{equation}

For $r^*=1$ we have $\V_{r^*}=\emptyset$ and $c_rs(\V_{r^*})<1$ for all $r\in \W$, i.e. \eqref{kkkkk} trivially holds. 

Consider $r^*>1$. As in the previous proof, we have $r^*\le K$.

{\em Sufficiency:} Assume that $c_r\,s(\V_{r^*})\ge 1$ and $r \not\in \V_{r^*}$, i.e. $r \geq r^*$. By \footnotesize Step 4 \normalsize of {\em SGa}, we have $c_{r^*}s(\V_{r^*})<1$. Since  $(c_r)_{r=1,\ldots,K}$ is non-increasing, it follows that  $c_r\,s(\V_{r^*})<1$ for $r \geq r^*$, which is a contradiction. Thus, $r\in \V_{r^*}$.  

{\em Necessity:}  For $r\in \V_{r^*}$, in view of \footnotesize Step 4 \normalsize of {\em SGa}, we have $c_r\,s(\V_r)\ge 1$. Thus, the second inequality in \eqref{ximo}  is satisfied with  $\A = \V_r$ and $\B = \{r\}$. Since $\V_r\cup\{r\}\subsetneq \W$ and $\V_r\cap\{r\}=\emptyset$, by Lemma \ref{prop:ximo},  the first inequality in \eqref{ximo} follows. Consequently, 
$$
s(\V_1)\le \ldots\le s(\V_{r^*}).
$$
Hence $c_r\,s(\V_{r^*})\ge 1$ for  $r\in \V_{r^*}$. 
\end{proof}

Finally, we propose a new algorithm, {\em coma} (named after \underline{c}hange \underline{o}f \underline{m}onotonicty \underline{a}lgorithm), which is a  modification of the approach from Stenger and Gabler (2005). To describe this algorithm, it is convenient to denote $s(\W)=0$. The set of strata labels $\W$ needs ordering as in {\em SGa}.  The algorithm {\em coma} proceeds as follows: 

\begin{algorithmic}[1]
	\State Reorder $\W = \{1, \ldots, K\}$ according to $ c_1 \ge \ldots \ge c_K$.
	\State Let $r = 1$.
	\State Compute $s(\V_r)$ and  $s(\V_{r+1})$ according to \eqref{esv}. 
	\State If $s(\V_r) > s(\V_{r+1})$, set $r^*=r$ and go to \footnotesize Step 5\normalsize;
	\Statex \hspace*{9.5pt} otherwise, let $r \gets r + 1$, and go to \footnotesize Step 3\normalsize.	
	\State Return $\x^{\V_{r^*}}$ according to \eqref{xop}.
\end{algorithmic}

\begin{proposition}\label{prop:coma}
The algorithm coma solves Problem \ref{probl}.
\end{proposition}  

\begin{proof}
	The proof follows from the equivalence \eqref{ximo} after referring to Proposition \ref{prop:SGa}. 
\end{proof}

Table \ref{tab00:} shows how {\em SGa} and {\em coma} work for the artificial population with 20 strata  considered in Table \ref{tab0:}. 

\begin{table}
	\centering 	\footnotesize
	\begin{tabular}{|r|r|r|r|r|r|r|}
		\hline
		$w$ & $r$ & $c_r$ & $s(\V_r)$ & $s(\V_{r+1})$ & $c_rs(\V_r)$  & $s(\V_r)/s(\V_{r+1})$ \\
		\hline
		{\bf 6} &	1 & 15.45  & 0.18736 & 0.25690 & 2.89497 & 0.72931   \\
		{\bf 17} &	2 & 10.21   & 0.25690  & 0.35221  & 2.62357 & 0.72939 \\
		{\bf 15} &	3 & 4.25  & 0.35221 &   0.39105 & 1.49666  &  0.90068  \\
		{\bf 2} &	4 & 2.65  &  0.39105 & 0.39115  &  1.00107  & 0.99974  \\
		11 &	5 &  2.37   & 0.39115 & 0.38190  & {\bf 0.92728} & {\bf 1.0242}\\
		\hline
	\end{tabular}
	\caption{\label{tab00:} \footnotesize For the artificial population considered in Table \ref{tab0:} we apply {\em SGa} and {\em coma} ($w$ refers to strata labels  as in Table 1, i.e. before ordering performed in Step 1 of both algorithms). In every iteration 1--4, one stratum is assigned to the set of {\em take-all} strata: 6, 17, 15 and 2. Both procedures stop at iteration 5.}
\end{table}

\section{Numerical experiments}

In the simulations, using the R software (2019), we compared the computational efficiency of {\em rNa}, {\em SGa} and {\em coma}  as well as some known algorithms for optimal allocation. Since the efficiency of {\em SGa} and {\em coma} turned out to be quite similar, we present results only for the {\em coma}. The R code used in our experiments is available at  \href{https://github.com/rwieczor/recursive_Neyman}{https://github.com/rwieczor/recursive\_Neyman}.

Two artificial populations with several strata were constructed by iteratively (100 and 200 times) binding collections of observations (each collection of 10000 elements) generated independently from lognormal distributions with varying parameters. The logarithms of generated random variables have mean equal to 0 and standard deviations equal to $\log (1+i), \; i=1,\dots, N_{max}$, where $N_{max}$ is equal  100 and 200, respectively.  
In each iteration, the strata were created using the geometric stratification method of Gunning and Horgan (2004) (implemented in the R package {\em stratification}) with parameter 10 being the number of strata and targeted coefficient of variation equal to 0.05.

For these populations, we calculated the following vectors of parameters: populations sizes, $(N_w)_{w\in\W}$, and population standard deviations in strata, $(S_w)_{w\in\W}$, both needed for the allocation algorithms. Finally, the original order of strata  was rearranged by a random permutation. In this way  two populations with 507 and 969 strata were created (for details of the implementation see the R code available on GitHub reprository).  

We used the {\em microbenchmark} R package for numerical comparisons of computational efficiency of the algorithms. The results obtained are presented in Fig. \ref{fig2}. The simulations  suggest that the {\em rNa} is typically more efficient than  {\em coma} (and  {\em SGa}). However, this is not always the case as Fig. \ref{fig3} shows. For the experiment referred to in Fig. \ref{fig3}, we created an artificial data set with significant differences in standard deviations between the strata. 

\begin{figure}
	\begin{center}
		\includegraphics[height=100mm,width=160mm]{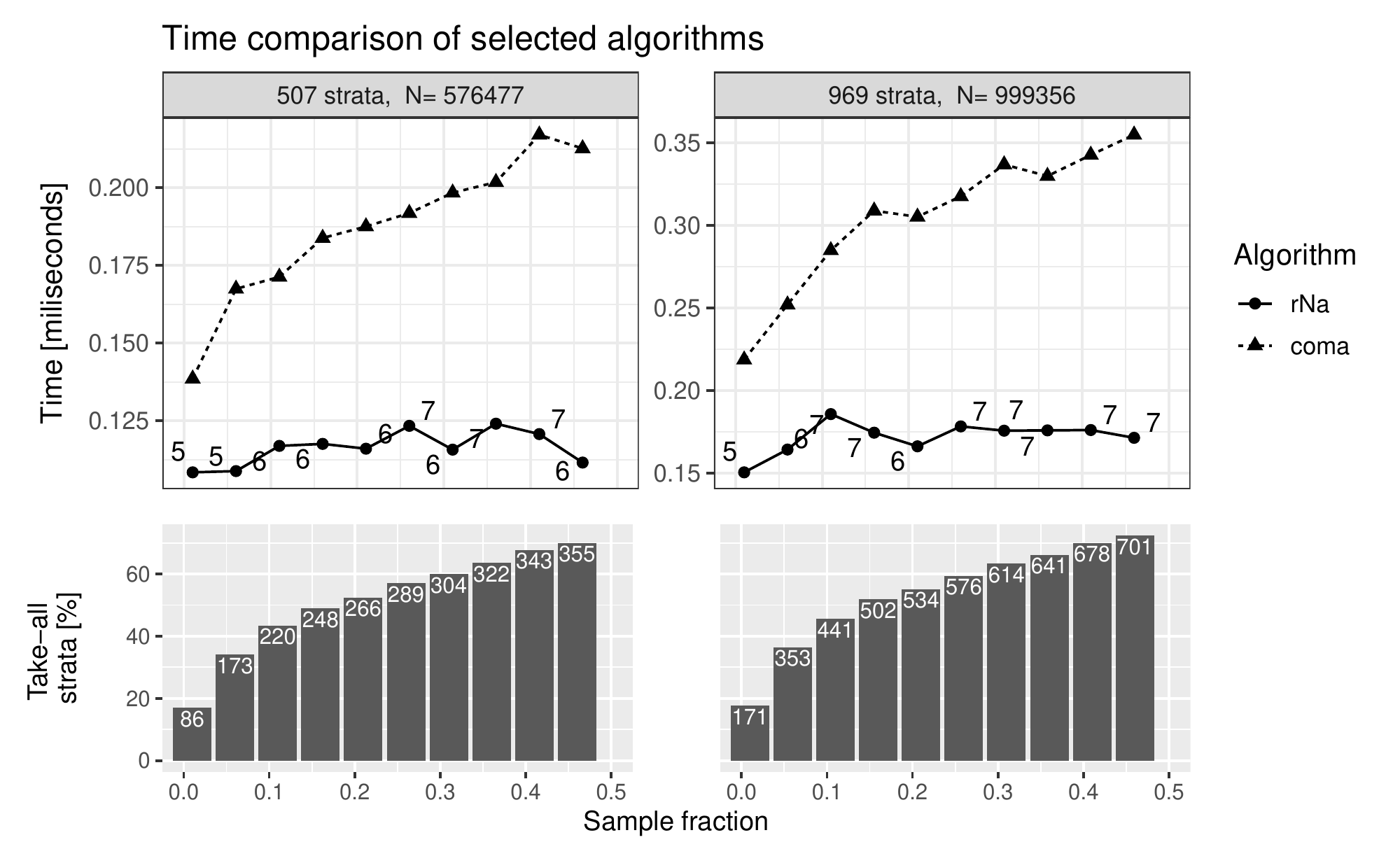}
		\caption{\label{fig2} 	\footnotesize Comparison of running times of R-implementations {\em coma} and {\em rNa} for two lognormal populations. Top graphs show the empirical median of performance times calculated from 100 repetitions. Numbers of iterations of the {\em rNa} follow its graph. Counts of {\em take-all} strata are inside bars of bottom graphs.}
	\end{center}
\end{figure}

\begin{figure}
	\begin{center}
		\includegraphics[height=100mm,width=160mm]{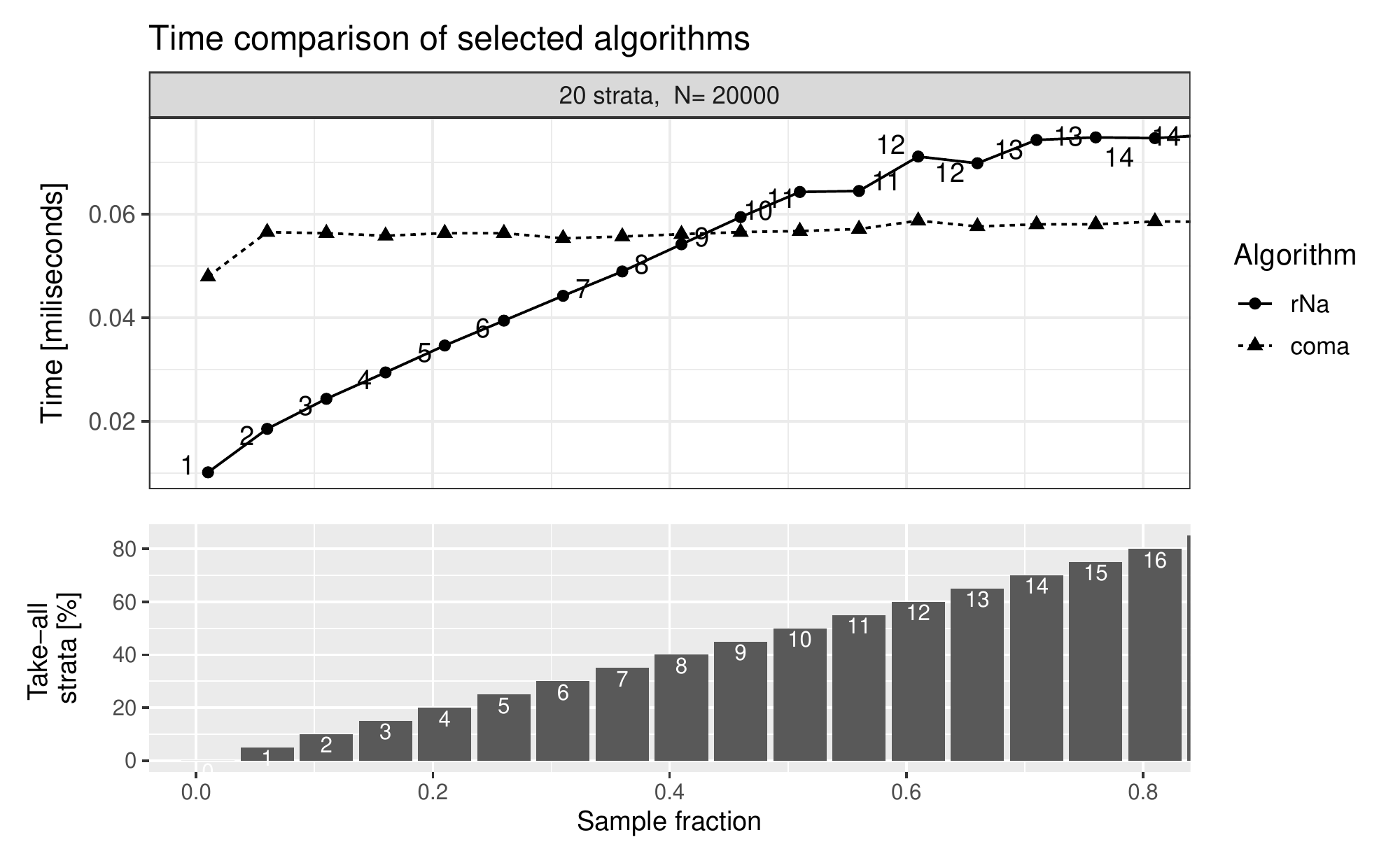}
		\caption{\label{fig3} 	\footnotesize Comparison of running times of R-implementations of  {\em coma} and {\em rNa}. Top graphs show the empirical median of performance times calculated from 100 repetitions. Computations were done for an artificial population with $S_w=10^w$, and $N_w=1000$, $w=1,\ldots,20$. Numbers of iterations of the {\em rNa} follow its graph. Counts of {\em take-all} strata  are inside bars of bottom graphs: the first two are $0,1$. }
	\end{center}
\end{figure}


We also compared computational efficiency of {\em rNa} and {\em coma} with the algorithms: {\em noptcond} of Gabler et al. (2012) and {\em capacity scaling}  of Friedrich et al. (2015), designed for, respectively, non-integer and integer allocation under both lower and upper bounds for the sample strata sizes.   Numerical experiments show that these two algorithms (with lower bounds set to zero) were considerably slower than  {\em rNa} and {\em coma}. Moreover, the design variances obtained for the optimal non-integer allocation before and after rounding (we used optimal rounding of Cont and Heidari (2014)) and for the optimal integer allocation, were practically indistinguishable, see Table  \ref{tab:3}.   

One may argue that in real life applications the number of strata may not be large and so differences in computational efficiency  are of marginal importance. Nevertheless, in some applications, like census-related surveys, the number of strata can be counted even in tens of thousands (there were more than 20 000 strata in German Census 2011, see Burgard and M\"unnich (2012)). The issue of computational complexity of optimal allocation algorithms has been addressed e.g. in M\"unnich et al. (2012). Computational efficiency  of optimal allocation algorithms becomes a crucial issue in iterative solutions of stratification problems, see e.g. Lednicki and Wieczorkowski (2003), Baillargeon and Rivest (2011) and Barcaroli (2014). In such procedures the allocation routine may be typically repeated a very large number of times (e.g. millions or more, depending on desired accuracy of approximations). 

\begin{table}[t]
	\centering 	\footnotesize
	\begin{tabular}{|r|r|r|r|r|}
		\hline
		sample  & \multicolumn{2}{|c|}{$K=507$} & \multicolumn{2}{|c|}{$K=969$} \\
		\cline{2-5}
		fraction	& $n$ & $D^2/D_0^2$  & $n$ & $D^2/D_0^2$ \\
		\hline
		 0.1 & 57648 & 0.999726 & 99936 & 0.999375\\
		\hline
		 0.2 & 115295 & 0.999964 & 199871 & 0.999948\\
		\hline
		 0.3 & 172943 & 0.999993 & 299807 & 0.999987 \\
		\hline
		 0.4 & 230591 & 0.999996 & 399742 & 0.999996 \\
		\hline
		 0.5 & 288238 & 0.999999 & 499678 & 0.999999\\
		\hline
	\end{tabular}
	\caption{\label{tab:3} 	\footnotesize Variances $D^2$ and $D_0^2$ are based on optimal non-integer and optimal integer allocations respectively. For variances $\tilde{D}^2$, based on rounded optimal non-integer allocation,  we systematically get $\tilde{D}^2/D^2_0=1$ (up to five decimal digits). }
\end{table}

\section{Final remarks and conclusions}
We proved that the recursive Neyman algorithm, {\em rNa}, is optimal under upper bounds on sample strata sizes. The approach is based on optimality of the $\V$-allocation \eqref{xop} and \eqref{asop}, derived from the KKT conditions.  We also proposed a modification, {\em coma}, of the algorithm, {\em SGa}, of Stenger and Gabler (2005) and, using the $\V$-allocation, we gave short proofs of optimality for both algorithms. 

Simulation comparisons of  computational efficiency showed that {\em rNa} is much faster than  {\em coma} and {\em SGa} (the latter two being of similar efficiency) and its relative efficiency increased with the sample fraction.  Nevertheless, there may exist  situations when {\em coma} (or {\em SGa}) happen to be more efficient than  {\em rNa}.  Ready-to-use R-implementations of the algorithms are available on CRAN repository in the new package {\em stratallo}:\newline   \href{https://cran.r-project.org/web/packages/stratallo}{https://cran.r-project.org/web/packages/stratallo}.

\vspace{5mm}\noindent\small
{\bf References} 
\begin{enumerate}
	\item  Bailargeon, S., and Rivest, P. (2011), "The construction of stratified designs in R with the package stratification," {\em Survey Methodology}, 37(1), 53-65.
	
	\item Barcaroli, G. (2014), "SamplingStrata: An R Package for the Optimization of Stratified Sampling," {\em Journal of Statistical Software}, 61(4), 1-24. 
	
	\item Boyd, S., and Vandenberghe, L. (2004), {\em Convex Optimization}, Cambridge Univ. Press.
	
	\item Burgard, J.P., and M\"unnich, R. (2012) "Modelling over and undercounts for design-based Monte Carlo studies in small area estimation: An application to the German register-assisted census," {\em Computational Statistics \& Data Analysis}, 56(10), 2856-2863.
	
	\item Cont, R., and Heidari, M. (2014), "Optimal rounding under integer constraints," {\em arXiv}, 1501.00014, 1-14.
	
	\item Friedrich, U., M\"unnich, R., de Vries, S., and  Wagner, M. (2015) "Fast integer-valued algorithm for optimal allocation under constraints in stratified sampling," {\em Computational Statistics and Data Analysis}, 92, 1-12.
	
	\item Gabler, S., Ganninger, M., and M\"unnich, R. (2012), "Optimal allocation of the sample size to strata under box constraints," {\em Metrika}, 75(2), 151-161.
	
	\item Gunning, P., and Horgan, J.M. (2004) "A new algorithm for the construction of stratum boundaries in skewed populations," {\em Survey Methodology} 30(2), 159-166.
	
	\item Kadane, J. (2005), "Optimal dynamic sample allocation among strata," {\em Journal of Official Statistics}, 21(4), 531-541.
	
	\item M\"unnich, R.T., Sachs, E.W., and  Wagner, M. (2012), "Numerical solution of optimal allocation problems in stratified sampling under box constraints," {\em Advances in Statistical  Analysis}, 96, 435-450.
	
	\item Neyman, J. (1934), "On the two different aspects of the representative method: the method of stratified sampling and the method of purposive selection," {\em Journal of the Royal Statistical Society}, 97, 558-625.

	\item R Core Team (2019), "R: A language and environment for statistical computing. R Foundation for Statistical Computing," [online]. Available at https://www.R-project.org/.
	
	\item Lednicki, B., and Wieczorkowski, R. (2003), "Optimal stratification and sample allocation between subpopulations and strata," {\em Statisics in Transition}, 6(2), 287-305.
	
	\item S\"arndal, C.-E., Swensson, B., and Wretman, J. (1992), {\em Model Assisted Survey Sampling},  New York, NY: Springer.
	
	\item Stenger, H., and Gabler, S. (2005), "Combining random sampling and census strategies - Justification of inclusion probabilities equal 1," {\em Metrika}, 61, 137-156.
	
	\item Statistics Canada (2010), {\em Survey Methods and Practices} [online]. Available at \newline https://www150.statcan.gc.ca/n1/en/pub/12-587-x/12-587-x2003001-eng.pdf 
	
	\item Tchuprov, A. (1923), "On the mathematical expectation of the moments of frequency distributions in the case of corelated observations," {\em Metron}, 2, 461-493. 
	
	\item Valliant, R., Dever, J.A., and Kreuter, F., (2018), {\em Practical Tools for Designing and Weighting Sample Surveys} (2nd ed.),  New York, NY: Springer.
	

	\item Wright, T. (2017), "Exact optimal sample allocation: More efficient than Neyman," {\em Statistics and Probability Letters}, 129, 50-57.
	
	\item Wright, T. (2020), "A general exact optimal sample allocation algorithm: With bounded cost and bounded sample sizes," {\em Statistics and Probability Letters}, 165/108829, 1-9.
\end{enumerate}

\appendix
\section{Appendix: Convex optimization scheme and the KKT conditions}
\begin{proof}[Proof of Theorem \ref{gener}.] 
Problem \ref{probl} belongs to a class of optimization problems of the form: minimize a strictly convex function $f:(0,\infty)^m\to \R$,  under  constraints 
$$
g_i(\x)\le 0,\quad  i=1,\ldots,r,\qquad \mbox{and}\qquad h_j(\x)=0,\quad j=1,\ldots,s,
$$
satisfied for all $\x\in(0,\infty)^m$, where $g_i$, $i=1,\ldots,r$, are convex and $h_j$, $j=1,\ldots,s$, are affine. It is well known, see e.g. Boyd and Vanderberghe (2004), that  in such case there exists a unique  $\x^*\in(0,\infty)^m$, such that $f$ attains its global minimum at $\x=\x^*$. The minimizer, $\x^*$, can  be identified through the set of equations/inequalities, known as the KKT conditions: 

There exist  $\lambda_i\in\R$, $i=1,\ldots,r$, and $\mu_j\in\R$, $j=1,\ldots,s$, such that
\bel{KTT1}  
\nabla f(\x^*)+\sum_{i=1}^r\,\lambda_i\nabla g_i(\x^*)+\sum_{j=1}^s\,\mu_j\nabla h_j(\x^*)= \0,
\ee 
and 
\bel{KTT2}
h_j(\x^*)=0, \quad\lambda_i\ge 0,\qquad \lambda_ig_i(\x^*)=0,\qquad 
g_i(\x^*)\le 0
\ee
for $i=1,\ldots,r$,  $j=1,\ldots,s$,

  We consider the KKT scheme with the objective function $f$ defined in \eqref{function} and with
	$$h(\x) = \sum_{w \in \W} x_w - n \qquad \mbox{and} \qquad g_w(\x)= x_w-b_w, \quad w\in \W. $$ 
	Thus $$\nabla f(\x)=-\left(\tfrac{a_w^2}{x_w^2}\right)_{w\in \W},\qquad  \nabla h(\x)=(1,\ldots,1),\qquad \nabla g_w(\x)={\mathbf 1}_w,\quad w\in \W,$$
	where $\mathbf 1_w$ denotes the vector with all entries 0, except the entry with label $w$ which is 1.
	
	Then, \eqref{KTT1} and \eqref{KTT2} read: there exist $\lambda_w\ge 0$, $w\in \W$, and $\mu\in\R$ such that
	\bel{KKT1}
	-\tfrac{a_w^2}{{x_w^*}^2}+\lambda_w+\mu=0,\quad w\in \W,
	\ee
	\bel{KKT2}
	\sum_{w\in \W}\,x_w^*=n,\quad \quad x_w^*\le b_w,\quad w\in \W,
	\ee
	and 
	\bel{KKT3}
	\lambda_w(x_w^*-b_w)= 0,\quad w\in \W.
	\ee
	
	Since Problem \ref{probl} is a convex optimization problem, its solution exists and is unique. Therefore, to prove Theorem \ref{gener} it suffices to show that conditions \eqref{KKT1} -- \eqref{KKT3} are satisfied for $\x^*=\x^\V$ with $\V$ defined in \eqref{asop}.
	
	Let $\mu=s^2(\V)>0$, with $s(\V)$ defined in \eqref{esv}, $\lambda_w=c_w^2-\mu$, $w\in \V$, and  $\lambda_w=0$ for $w\not\in  \V$. Note that  \eqref{asop} yields $\lambda_w\ge 0$ for $w\in \V$. Then \eqref{KKT1}, as well as the equality in \eqref{KKT2}, i.e.
	$$
	\sum_{w\in \W}\,x_w^\V=\sum_{w\in \V}\,b_w+s(\V)\sum_{w\not\in \V}\,a_w=n,
	$$ 
	are satisfied. 	Inequalities in \eqref{KKT2}  are  trivial for $w\in \V$, and for $w\not\in \V$  they follow from \eqref{asop}.  Finally, \eqref{KKT3} holds true, since  $\lambda_w=0$ for $w\not\in \V$ and  $x_w^\V-b_w=0$ for $w\in \V$.
\end{proof}

\end{document}